\documentclass[11pt, reqno]{amsart}

\usepackage{amsmath,amsfonts, amsthm, amssymb, color, cite}

\usepackage{fancyhdr}

\usepackage{lastpage}

\usepackage{cases}
\newtheorem{thm}{Theorem}[section]
\newtheorem{lma}{Lemma}[section]
\newtheorem{pro}{Proposition}[section]

\theoremstyle{definition}

\theoremstyle{remark}

\numberwithin{equation}{section}
\allowdisplaybreaks

\def\f{\frac}

\def\hf1{^\f{1}{1-\xi^2}}

\def\be{\begin{equation}}
\def\en{\end{equation}}
\def\bs{\begin{split}}
\def\es{\end{split}}
\def\ba{\begin{align}}
\def\ea{\end{align}}

\author[Lin Chang]{L\MakeLowercase{in} C\MakeLowercase{hang}}
\address{School
of Mathematics   Science, Beihang University, Beijing, China}
\email{changlin23@buaa.edu.cn}
\renewcommand{\fancyhead}{}
\title{ S\MakeLowercase{tability} \MakeLowercase{of} \MakeLowercase{a}   C\MakeLowercase{omposite}
w\MakeLowercase{ave}  \MakeLowercase{of} t\MakeLowercase{wo}  s\MakeLowercase{eperate} s\MakeLowercase{trong}   v\MakeLowercase{iscous}  s\MakeLowercase{hock}  w\MakeLowercase{aves} \MakeLowercase{for}  1-D  i\MakeLowercase{sentropic}  n\MakeLowercase{avier}-s\MakeLowercase{tokes}  s\MakeLowercase{ystem}  }
\keywords{combination of two shock waves ; energy estimate; asymptotic stability}
\date{\today}
\begin{document}
\begin{abstract}
 In this paper,   the large time behavior of solutions   of 1-D isentropic Navier-Stokes system is  investigated. It is shown that a composite wave consisting of two viscous shock waves  is stable for the Cauchy problem provided that the two waves are  initially far away from each other. Moreover  the strengths of two   waves could be arbitrarily large.
\end{abstract}
\maketitle
\section{Introduction}
We consider the following  one-dimensional   isentropic Navier-Stokes system  for polytropic gas in the Lagrangian coordinate,
\begin{equation}\label{1.1}
\left\{ \begin{array}{ll}
&v_t-u_x=0,\quad \quad \quad\quad \quad \quad t>0, x\in \mathbb{R},\\
&u_t+p_x=(\mu(v)\frac{u_x}{v^{ }})_x,    \quad \quad  \, \,t>0, x\in \mathbb{R},
\end{array} \right.
\end{equation}
with the initial data:
\begin{equation}\label{1.2}
(v, u)(x, 0) = (v_0,u_0) (x)\longrightarrow (v_{\pm},u_{\pm}),\quad \text{as} \quad  x\rightarrow \infty.
\end{equation}
Here $v(x,t)=\frac{1}{\rho(x,t)}$ is the specific volume, $u(x,t)$   the fluid velocity, $p=a v^{-\gamma}$  the pressure with constant $a>0$, $\gamma> 1$  the adiabatic constant, and $\mu(v)=\mu_{0}v^{-\alpha}$   the viscosity coefficient with $\alpha\geq 0$. Without loss of generality, we assume $\mu_0=1$ in the rest of this article.
When the viscosity $\mu(v)\equiv 0$, the system (\ref{1.1}) becomes  the  Euler system
\begin{equation}
\left\{ \begin{array}{ll}\label{1.3}
&v_t-u_x=0,\\
&u_t+p_x=0.
\end{array} \right.
\end{equation}

It is known that the equation (\ref{1.3}) has rich  wave phenomena, such as shock and rarefaction.   The shock  is mollified  as the so-called viscous shock wave  when $\mu(v)>0$. The    time asymptotic  stability   of  single wave pattern   has been extensively studied in a large amount of literature since the pioneer works of \cite{g1986,mn1985}, see   \cite{fs1998,  hlz2017,km1985,l1997,lz2009,lz2015,m,mn1994,sx1993} and the reference therein, see other interesting works on the composite wave \cite{hm2009,hlm2010}.
However, most of above works require the strength of shock wave is   small, that is, the shock is weak. The stability of large amplitude shock (strong shock) is more interesting and challenging in both mathematics and physics.

Matsumura-Nishihara \cite{mn1985} showed that the viscous shock wave is stable if $ |v_+-v_-|<C(\gamma - 1)^{-1}$, that is, when $ \gamma  $ tends   to  $1  $, the strength of shock wave could be large. The condition is later relaxed to the condition that $ |v_+-v_-|<C(\gamma - 1)^{-2}$ in \cite{km1985}. The restriction on the strength of shock  was removed in  \cite{mw2010}   by a wonderful weighted energy method as $\alpha >\frac{\gamma-1} {2} $. Vasseur-Yao  \cite{vy2016} removed the condition $\alpha >\frac{\gamma-1} {2}$ by introducing an elegant variable transformation.  Moreover,  He-Huang  \cite{hh2020} extended the result of \cite{vy2016} to general pressure $p(v)$ and  viscosity $\mu(v)$, where $\mu(v)$ could be any positive smooth function.

It is important to study the stability of   composite wave consisting of at least two waves.  From \cite{hm2009},  it is not difficult to show the asymptotic stability of a composite wave consisting of 1-viscous shock wave and 2-viscous shock wave, provided that the strengths of the two shocks   satisfy  the condition ``small with same order".

In this paper, we study the asymptotic stability of this kind of composite wave with two large amplitude shock waves. More precisely, let $(V_1, U_1 )(x, t) $ be the 1-viscous shock wave connecting the left state $(v_-, u_- ) $ with an intermediate state $(v_m, u_m  ) $ and $(V_2, U_2 ) (x, t)$  be the 2-viscous shock wave connecting  $(v_m, u_m ) $ with the right state $(v_+, u_+ ) $ where the viscous shock waves are given in (\ref{2.2}) and (\ref{2.3}). The  intermediate state  $(v_m, u_m ) $  is determined by the RH condition, i.e.,
\begin{equation}\label{1.4}
\left\{\begin{array}{ll}
-s_{2}(v_+-v_m)-(u_+-u_m)=0, \\
-s_{2}(u_+-u_m)+(p(v_+)-p(v_m))=0,
\end{array}
\right.
\end{equation}
and
\begin{equation}\label{1.5}
\left\{\begin{array}{ll}
-s_{1}(v_m-v_-)-(u_m-u_-)=0, \\
-s_{1}(u_m-u_-)+(p(v_m)-p(v_-))=0.
\end{array}
\right.
\end{equation}
We denote the composite wave consisting of the two viscous shock waves $(V_{i}  ,U_{i}  ),$    $i = 1, 2$  by  $( V, U)(x, t) = (V_{1}  + V_{2} - v_{m},U_{1 }+U_{2}-u_{m} )$.

We outline the strategy as follows. In order to remove the condition ``small with same order", motivated by \cite{vy2016} and \cite{hh2020}, we introduce a new variable ${}{h}$, and formulate a new equation $\eqref{3.2}_2$ for ${}{h}$ in which the viscous term is moved to the mass equation $\eqref{3.2}_1$ such that the two nonlinear terms ${}{p}_x$ and $(\frac{{}{v}_x}{{}{v}^{\alpha+1}})_x $ are decoupled, so the interaction between nonlinear terms is weaken, and  the    low order   estimates are obtained. We then turn to the original system \eqref{1.1} to derive the higher order energy estimates, and finally complete the a priori estimates. On the other hand, since  the strengths of 1-shock wave and 2-shock  wave are arbitrarily large, the interaction between the two shocks is strong. We have to assume that 1-shock  wave is initially far away from 2-shock wave so that the interaction is weak.

The rest of the paper will be arranged as follows. In section \ref{section2}, the composite wave is formulated and the main result is stated.   In section \ref{section3}, the problem is reformulated by the anti-derivatives of the perturbations around the  composite wave. In section \ref{section4}, the a priori estimates are established. In section \ref{section5}, the main theorem is proved.

\noindent \textbf { Notation.}
The functional $\|\cdot\|_{L^p(\Omega)}$ is defined by $\| f\|_{L^p(\Omega)} =  (\int_{\Omega}|f|^{p}(\xi )\operatorname{ d }\xi )^{\frac{1}{p}}$. The symbol $\Omega$ is often omitted  when $\Omega=(-\infty,\infty)$. We denote for simplicity
\begin{equation*}
\| f\| =  \left(\int_{ -\infty}^{ \infty}f^{2}(\xi )\operatorname{ d }\xi \right)^{\frac{1}{2}}
\end{equation*}
as $p=2$. In addition, $H^m$ denotes the  $m$-th  order Sobolev space of functions defined by
 \begin{equation*}
\|f\|_{m} =  \left( \sum_{k=0}^{m}  \|\partial^{k}_{\xi}f\|^2 \right)^{\frac{1}{2}}.
\end{equation*}{}

\section{Preliminaries and Main Theorem }\label{section2}
\subsection{ Viscous Shock Profile}
Before stating the main results, we recall the Riemann problem for the   Euler equation (\ref{1.3}) with the Riemann initial data
\begin{equation}\label{Riemann}
(v,u)(x,0)=\left\{ \begin{array}{ll}
&(v_-,u_-),x<0,\\
&(v_+,u_+),x>0.\\
\end{array} \right.
\end{equation}
It is known that the system (\ref{1.3}) has two eigenvalues: $\lambda_{1}=-\sqrt{- p'  (v)}<0 $, $\lambda_{2}=-\lambda_{1}>0 $.  By the standard arguments (e.g. \cite{s1983}),     we define the shock curve  $S_{1}$ (resp $S_2$)\\
\begin{equation*}
 \left\{ \begin{array}{ll}
u=u_--\sqrt{(v_--v)(v^{-\gamma}-v_{-}^{-\gamma})}  ,u<u_{-} ,v<v_{-}  ,S_1,\\
u=u_--\sqrt{(v_--v)(v^{-\gamma}-v_{-}^{-\gamma})}  ,u<u_{-} ,v>v_{-}  ,S_2,\\
\end{array} \right.
\end{equation*}\\
and $SS(v_-,u_-)$:
\begin{eqnarray*}
SS(v_-,u_-)=     \{ (v,u)|  u \leq  u_{-} ;  S_{1}(u) < v <  S_{2}(u)  \}.
\end{eqnarray*}
In this paper, we assume that $(v_+,u_+)\in SS (v_{-},u_{-})$.    Thus the Riemann solution of (\ref{1.3}),(\ref{Riemann}) consists of two shock waves (and three constant states), that is,  there exists an intermediate state $\left(v_{m}, u_{m}{}\right)$, such that $(v_m,u_m)\in S_1(v_{-},u_{-})$ with the shock speed $s_{1}<0$, and $(v_{ {+}},u_{ {+}})\in S_2(v_{m},u_{m})$ with the shock speed $s_{2}>0.$ Here the shock speeds $s_{1}$ and $s_{2}$ are constants determined by the $\mathrm{RH}$ condition and satisfy entropy conditions
\begin{eqnarray}
\lambda_{1}\left(v_{-} {}\right)>s_{1}>\lambda_{1}\left(v_{m} {}\right), \lambda_{2}\left(v_{m} \right)>s_{2}>\lambda_{2}\left(v_{+} \right).
\end{eqnarray}
In what follows, we define $( \chi_{1}, \chi_{2} )$ below
\begin{eqnarray*}
\chi_{1}: = v_{-}-v_{m},\qquad \chi_{2}:= v_{+}-v_{m}.
\end{eqnarray*}
  We see that the   1-shock wave is a traveling wave solution of (\ref{1.1}) with the formula $(V_{1},U_{1} )(x-s_{1}t)$, satisfying
\begin{equation}\label{2.2}
\left\{ \begin{array}{ll}
&{-s_1}{V_1}'-{U_1}'=0,\\
&{-s_1}U_1'+p(V_1)'=(\frac{U_1'}{V_1^{\alpha+1}})',\\
&(V_{1},U_{1})(+\infty)=(v_m,u_{m}), \\
&(V_{1},U_{1})(-\infty)=(v_-,u_-),
\end{array} \right.
\end{equation}
where $'= {\operatorname{d}}/{\operatorname{d}\xi_{1}}, \quad \xi_1=x-s_1t .$ Similarly, the 2-viscous shock wave $(V_{2},U_{2} )(x -s_{2}t)$  satisfies
\begin{equation}\label{2.3}
\left\{ \begin{array}{ll}
&{-s_2}{V_2}'-{U_2}'=0,\\
&{-s_2}U_2'+p(V_2)'=(\frac{U_2'}{V_2^{\alpha+1}})',\\
&(V_{2},U_{2})(+\infty)=(v_{+},u_{+}),\\
&(V_{2},U_{2})(-\infty)=(v_m,u_m),
\end{array} \right.
\end{equation}
where $'= {\operatorname{d}}/{\operatorname{d}\xi_{2}}, \quad \xi_2=x-s_2t .$

\begin{lma}\label{lemma2.1} (\cite{km1985}) There are positive constants $C  $ and  $c_{1,2} ,$  such that
\begin{flalign*}
\begin{split}
&  (U_{i})_{x}\leq 0, \quad i= 1,2,   \\
&\left| V_{1}-v_{m}  \right|\leqslant C \chi_{1} \mathrm{e}^{-c_{1}\left|x-s_{1} t\right|}, \quad x>s_{1} t, \quad t \geqslant 0, \\
&\left| V_{2}-v_{m}  \right|\leqslant C  \chi_{2} \mathrm{e}^{-c_{2}\left|x-s_{2} t\right|}, \quad x<s_{2} t, \quad t \geqslant 0.
\end{split}
\end{flalign*}
\end{lma}

\subsection{Location of the Shift $\beta_{1}$ and $\beta_{2}$ }
As mentioned before, we assume that the 2-viscous shock wave is initially far away from the 1-viscous shock profile, that is, the shock profile is  $(V_2,U_2)(x-\beta)$ with some constant $\beta>0$ as $t=0$. The two shocks formulate a composite wave by $V(x)=V_{1}(x )+V_{2}(x -\beta_{ })-v_{m}, U(x)=U_{1}(x )+U_{2}(x  -\beta_{ })-u_{m}$. We
consider the situation where the initial data $({}{v}_0,{}{ u}_0)(x)$ is given in a neighborhood of $(V,U)(x)$. The solution is expected to tend to the composite wave
 \begin{align}\label{2.4}
\begin{split}
&V(x,t;\beta_1,\beta_2;\beta)=V_{1}(x-s_{1}t+ \beta_{1})+V_{2}(x -s_{2}t-\beta_{ }+\beta_{2})-v_{m},   \\ &U(x,t;\beta_1,\beta_2;\beta)=U_{1}(x-s_{1}t+ \beta_{1})+U_{2}(x -s_{2}t-\beta_{ }+\beta_{2})-u_{m},
\end{split}
 \end{align}
where the shifts $\beta_1$ and $\beta_2$ are supposed to satisfy
\begin{align*}
\begin{split}
0=&\int_{-\infty}^{\infty}\left(
\begin{array}{cccc}
{}{v}_0(x) -    V(x,0;\beta_{1}, \beta_{2};\beta)  \\
{}{u}_0(x) -    U(x,0;\beta_{1}, \beta_{2};\beta)
\end{array}
\right )
\operatorname{d}x:=\left(\begin{array}{cccc}
I_1(\beta_1,\beta_2;\beta)  \\
I_2(\beta_1,\beta_2;\beta)
\end{array}
\right ).
\end{split}&
\end{align*}
We shall find unique $\beta_1$ and $\beta_2$ such that $I_i(\beta_1,\beta_2;\beta)=0, i=1,2.$
Note that
\begin{align}\label{2.5}
\begin{split}
&I_1(\beta_1,\beta_2;\beta)\\
=&\int_{-\infty}^{\infty}
{}{v}_0(x) -    V(x,0;0,0;\beta)+ V(x,0;0,0;\beta) -   V(x,0;\beta_{1}, \beta_{2};\beta)  \operatorname{d}x\\
=&I_{01}+\int_{-\infty}^{\infty} V_1(x ) -   V_1(x +\beta_1)  \operatorname{d}x\\
&+\int_{-\infty}^{\infty} V_2(x-\beta) -   V_2(x-\beta+\beta_2) \operatorname{d}x\\
=&I_{01}-\beta_1(v_m-v_-)-\beta_2(v_+-v_m),
\end{split}
\end{align}
where
\begin{equation}\label{2.6}
I_{01}=\int_{-\infty}^{\infty}
{}{v}_0(x) -    V(x,0;0,0;\beta)    \operatorname{d}x.
\end{equation}
Similarly one can get
\begin{align}
I_2(\beta_1,\beta_2;\beta)=I_{02}-\beta_1(u_m-u_-)-\beta_2(u_+-u_m),
\end{align}
where
\begin{equation}\label{2.8}
I_{02}=\int_{-\infty}^{\infty}
{}{u}_0(x) -    U(x,0;0,0;\beta)   \operatorname{d}x.
\end{equation}
Utilizing (\ref{2.5})-(\ref{2.8}), R-H condition (\ref{1.4})-(\ref{1.5}),  we have
\begin{align*}
\begin{split}
\left(
\begin{array}{cccc}
I_{01}  \\
I_{02}
\end{array}
\right )=  -\beta_{1}\left(
\begin{array}{cccc}
v_--v_{m}   \\
u_{-}-u_{m} \\
\end{array}
\right )-
\beta_{2}\left(
\begin{array}{cccc}
v_m-v_{+}   \\
u_m-u_{+} \\
\end{array}
\right ).\\
\end{split}&
\end{align*}
Thus, one gets
\begin{eqnarray}\label{2.9}
\beta_{1}    =\frac{I_{01}s_{2}+I_{02}}{\chi_{1}(s_{1}-s_2)  }, \quad \beta_{2} =\frac{I_{01}s_{1}+I_{02}}{\chi_{2}(s_{1}-s_2) }.
\end{eqnarray}

\subsection{Main Theorem }
To state the main theorem, we assume
\begin{flalign}\label{2.10}
\begin{split}
&v_{0}(x)-V(x,0;0,0;\beta)\in \mathbb{H}^1   \cap \mathbb{L}^1 \quad u_{0}(x)-U(x,0;0,0;\beta)\in \mathbb{H}^1    \cap \mathbb{L}^1.
\end{split}
 \end{flalign}
 Then we can define
\begin{eqnarray}
(\phi_{0},\psi_0)(x)= \int_{-\infty}^{x} [ v_{0}(y )-V(y,0;\beta_{1},\beta_{2};\beta)  ,  u_{0}(y )-U(y,0;\beta_{1},\beta_{2};\beta)    ] \operatorname{d}y.
\end{eqnarray}
In view of $I_i(\beta_1,\beta_2;\beta)=0, i=1,2$,  we  further assume that
\begin{eqnarray}\label{2.12}
(\phi_{0},\psi_0) \in \mathbb{L}^2.
\end{eqnarray}
We abbreviate $(V(x,t;\beta_{1},\beta_{2};\beta),U(x,t;\beta_{1},\beta_{2};\beta)) \ \text{as}\  (V ,U )$  in the rest part of this paper.
We are ready to state the main result.
\begin{thm}\label{theorem}
Suppose    (\ref{2.10})-(\ref{2.12}) hold and  {$(v_+,u_+)\in SS (v_{-},u_{-})$}.  There exists a positive constant   $\delta_0$, such that if $$\|\phi_{0}\|_{2}+\|\psi_{0}\|_{2}+\beta^{-1}\leq\delta_{0},$$ then the Cauchy problem (\ref{1.1}),(\ref{1.2})   has a unique global solution $(v,u)(x,t) $ satisfying
\begin{align}
\begin{split}
&( {v}-V,{u}-U)\in C^0([0,+\infty);\mathbb{H}^2), {v}-V \in \mathbb{L}^2( 0,+\infty ;\mathbb{H}^3),\\
&{u}-U\in \mathbb{L}^2( 0,+\infty ;\mathbb{H}^2).
\end{split}\label{2.13}
\end{align}
Moreover,
\begin{eqnarray}
\sup_{x\in \mathbb{R}_{ }}                       | {v }-V  |     \rightarrow 0 , \quad \sup_{x\in \mathbb{R}_{ }}                       | {u}-U  |     \rightarrow 0,    \text{   as } t\rightarrow +\infty.
\label{2.14}\end{eqnarray}
\end{thm}
\section{Reformulation of the Original Problem }\label{section3}
Set
 \begin{align}\label{3.1}
 \begin{split}
\phi&( x,t):=\int^x_ {-\infty}{}{(v-V)}(y,t)  \operatorname{d}y,  \\
  \psi&( x,t):=\int^x_{-\infty} {}{(u-U)}(y,t)  \operatorname{d}y,
\end{split}
 \end{align}
which means that we look for the solution $(v,u)(x,t)$ in the form
\begin{align}\label{t3.2}
v&( x,t)=\phi_x (x,t)+  V(x,t;\beta_{1},\beta_{2};\beta), \nonumber\\
u&( x,t)=\psi_x (x,t)+  U(x,t;\beta_{1},\beta_{2};\beta).
\end{align}
From (\ref{2.2})-(\ref{2.4}), we know the shock profile $(V,U)$ satisfies
\begin{equation}
\left\{ \begin{array}{ll}
&{V}_{t}-{U}_{x}=0,\\
&U_{t}+p(V)_{x}-\left(\frac{U_{x}}{V^{\alpha+1}}\right)_{x}={}{W}_{x},\\
&(V,U)(\pm \infty,t;\beta_{1},\beta_{2};\beta)=(v_{\pm}, u_{\pm}),
\end{array} \right.\label{3.2}
\end{equation}
where
\begin{equation*}
W=  \frac{U_{2x}}{V_2^{\alpha+1}}  + \frac{U_{1x}}{V_{1}^{\alpha+1}}  - \frac{U_{x}}{V^{\alpha+1}}+p(V)+p(v_m)-p(V_{1})-p(V_2).
\end{equation*}
Motivated by  \cite{mm1999}, substitute (\ref{t3.2}) into (\ref{1.1}) and integrate the resulting system with respect to $x$,  we have
\begin{equation}
\left\{ \begin{array}{ll}
&\phi_{t}-\psi_{x}  =0, \\
&\psi_{t}-f(V,U_x) \phi_{x}-\frac{{\psi_{x x}}}{V^{\alpha+1}} =F-W ,
\end{array} \right.  \label{3.3}
\end{equation}
 with the initial condition:
\begin{align}\label{3.4}
\left(\phi_{0}, \psi_{0}\right)&(x) \in \mathbb{H}^{2}, \quad x \in \mathbb{R},
\end{align}
where
\begin{align}\label{3.5}
f(V,U_x)= -p^{\prime}(V)    - (\alpha+1)\frac{{} U_{x}}{V^{\alpha+2}} >0,
\end{align}
\begin{align}\label{3.6}
F&= \frac{{}{u}_{x}}{{}{v}^{\alpha+1}}-\frac{U_{x}}{V^{\alpha+1}} - \frac{\psi_{xx}}{V^{\alpha+1}}+(\alpha+1)\frac{U_x \phi_x }{V^{\alpha+2}}-\left[p({}{v})-p(V)-p'(V)\phi_x\right].
\end{align}
We will seek the solution in the functional space $\mathbb{X}_{\delta}(0,T)$ for any $0\leq T < +\infty $,
\begin{align*}
\begin{split}
\mathbb{X}_{\delta}(0,T):=&\left\{ (\phi,\psi)\in C ([0,T];\mathbb{H}^2)|\phi_{x} \in  \mathbb{L}^2(0,T;\mathbb{H}^1) ,\psi_{x} \in  \mathbb{L}^2(0,T;\mathbb{H}^2)\right. \\
&  \sup_{0\leq t\leq T}\|(\phi, \psi)(t)\|_2\leq\delta    \},
\end{split}&
\end{align*}
 where ${ \delta} \ll 1$ is small.
\begin{pro}\label{prposition3.1} (A priori estimate)
Suppose  that $(\phi,\psi) \in \mathbb{X}_{\delta}(0,T)$ is the solution of  (\ref{3.3}), (\ref{3.4})  for some time $T>0$. There exists a positive constant $\delta_0  $ independent of  $T$, such that if
$$ \sup_{0\leq t\leq T}\|(\phi, \psi)(t)\|_{{2}} \leq \delta_{ }\leq \delta_{0},$$   for $t \in [0,T]$,
then
\begin{eqnarray*}
\|(\phi, \psi)(t)\|_{{2}}^{2}+    \int_{0}^t     (  \|  \phi_{x}(t)  \|^2_{1}  +   \|\psi_{x}(t)\|_{2}^2   )    \operatorname{d}t       \leq C_{0} ( \|(\phi_{0},\psi_{0})\|_{{2}}^2+   e^{-C_-\beta}),
\end{eqnarray*}
where $C_{0}  >1$ and $C_{-} $ are two positive constants independent of $T$.
\end{pro}

As long as  Proposition \ref{prposition3.1} is proved,   the local solution $(\phi,\psi)$   can be extend to $T =+\infty. $ We have the following Lemma.

\begin{lma}\label{lemma3.1}
  If $(\phi_{0},\psi_{0})\in \mathbb{H}^2$, there exists a positive constant $\delta_{1}=\frac{\delta_{0}}{\sqrt{C_{0}}}$, such that if
$$\|(\phi_{0},\psi_{0})\|_{{2}}^2+   e^{-C_-\beta} \leq \delta_{1}^{2},$$ then  the Cauchy problem  (\ref{3.3}), (\ref{3.4}) has a unique global solution
$(\phi,\psi)\in \mathbb{X}_{\delta_{0}}(0,\infty)$  satisfying
\begin{align*}
\sup_{t\geq0}\|(\phi, \psi)(t)\|_{{2}}^{2}+    \int_{0}^\infty   (  \|  \phi_{x}(t)  \|^2_{1}  +   \|\psi_{x}(t)\|_{2}^2   )    \operatorname{d}t       \leq C_{0}  ( \|(\phi_{0},\psi_{0})\|_{{2}}^2+   e^{-C_-\beta}).
\end{align*}
\end{lma}

{}

\section{a Priori Estimate  }\label{section4}
Throughout this section, we assume that the problem $(\ref{3.3}), (\ref{3.4})$ has a solution $(\phi,\psi)\in \mathbb{X}_{\delta}(0,T), $ for some $T>0$,
\begin{eqnarray}\label{4.1}
\sup_{0\leq t\leq T}\|(\phi, \psi)(t)\|_{2}\leq \delta.
\end{eqnarray}
It follows from the Sobolev inequality that $\frac{1}{2}v_{m}\leq v \leq \frac{3}{2} \max\{v_{
-},v_{+}\}$, and
\begin{align*}
\sup_{0\leq t\leq T}  \{  \|(\phi, \psi)(t)\|_{\mathbb{L}^{\infty}}    +   \|(\phi_{x}, \psi_{x})(t)\|_{\mathbb{L}^{\infty}}   \}  \leq  {\delta} .
\end{align*}

\subsection{  Low Order Estimate.}In order to  remove  the condition  ``small with same order", we introduce  a new perturbation $(\phi,\Psi)$ instead of $(\phi,\psi)$, where $\Psi$ will be  defined below. Motivated by \cite{vy2016} and \cite{hh2020}, we introduce a new effective velocity   ${}{h}={}{u}-{}{v}^{-(\alpha+1)}{}{v}_{x}$. Setting $ h_0 (x)=:h(x,0) $, the equations (\ref{1.1}), (\ref{1.2}) become
\begin{equation}
\left\{ \begin{array}{ll}
&{}{v}_t-{}{h}_x=(\frac{{}{v}_{x}}{{}{v}^{\alpha+1}})_{x},\\
&{}{h}_t+{}{p}_x=0,
\end{array} \right.\label{4.2}
\end{equation}
and
\begin{equation*}
(v_{0}, h_{0})(x) = (v_0,u_0-v_{0}^{-(\alpha+1)}{v}_{0x}) (x)\longrightarrow (v_{\pm},u_{\pm}),\quad \text{as} \quad  x\rightarrow \pm\infty.
\end{equation*}
Let $H=U-V^{-(\alpha+1)}V_{x}$. Then (\ref{3.2}) is equivalent to
\begin{equation}
\left\{ \begin{array}{ll}
&V_{t}-H_{x}=\left(\frac{V_{x}}{V^{\alpha+1}}\right)_{x},\\
&H_{t}+p(V)_{x}= {W}_{x},\\
&(V,H)(\pm \infty,t)=(v_{\pm},u_{\pm}).
\end{array} \right.\label{4.4}
\end{equation}\\
We define
\begin{eqnarray}
 \int_{-\infty}^x({}{h}-H)\operatorname{d}x:=\Psi.
\label{4.5}\end{eqnarray}
Substituting (\ref{4.4}) from (\ref{4.2}) and integrating the resulting system with respect to $x$, we have from  (\ref{4.5}), $(\ref{3.1})_{1}$ that
\begin{equation}
\left\{ \begin{array}{ll}
&\phi_t- \Psi_x-\frac{\phi_{xx}}{V^{\alpha+1}}+(\alpha+1)\frac{V_x \phi_x }{V^{\alpha+2}}=G\\
&\Psi_t+p'(V)\phi_x=-p({v}|V)-W,
\end{array} \right.  \label{4.6}
\end{equation}
where
\begin{eqnarray*}
G=\frac{{}{v}_{x}}{{}{v}^{\alpha+1}}-\frac{V_{x}}{V^{\alpha+1}} - \frac{\phi_{xx}}{V^{\alpha+1}}+(\alpha+1)\frac{V_x \phi_x }{V^{\alpha+2}},
\end{eqnarray*}
\begin{eqnarray*}
p({}{v}|V)=\left(p({}{v})-p(V)\right)-p'(V)\phi_x,
\end{eqnarray*}
with the  initial data
\begin{eqnarray*}
 \phi  (x,0) \in \mathbb{H}^{2}, \quad \Psi (x,0) \in \mathbb{H}^{1}.
\end{eqnarray*}
\begin{lma}\label{lemma4.1}
Under the assumption of (\ref{4.1}), it holds that
\begin{flalign}\label{s4.07}
\begin{split}
&|p({}{v}|V)|\leq C \phi_{x}^{2},\\
&|p({}{v}|V)_{x}|\leq C (|\phi_{xx}\phi_{x}|+|V_x|\phi_{x}^{2}),\\
&|G|\leq C (|\phi_{xx}\phi_{x}|+|V_x|\phi^{2}_{x}),
\end{split}
\end{flalign}
\begin{flalign}\label{s4.08}
\begin{split}
&|F|\leq C( \phi_{x} ^{2}+|\phi_{x}\psi_{xx}|),\\
&|F_{x}|\leq C(\phi_{x}^{2}+|\phi_{x}^{ } \phi_{x x}^{ }|+|\psi_{x x}^{ } \phi_{x x}^{ }|+|\psi_{x x x}^{ } \phi_{x}^{ }|+|\phi_{x}^{ } \psi_{x x}^{ } |,
\end{split}
\end{flalign}

and
\begin{eqnarray}\label{s4.09} \begin{split}
\|    {\Psi}_{0}  \|_{1}^{2}    \leq&  \|    {\psi}_{0}  \|_{1}^{2}+    C  \|    {\phi}_{0}   \|_{2}^{2},\\
 \| \psi \|^{2}  \leq &     \| \Psi \| ^2+ C\|\phi \|_{ 1}^2,      \\
 \| \psi_{x} \|^{2} \leq &     \| \Psi_{x} \| ^2+ C\|\phi_{x} \|_{ 1}^2.
\end{split}\end{eqnarray}
Here $C$ is a constant depends only on $  v_\pm $ and $  u_\pm $.
\end{lma}

\begin{proof}
Note that
\begin{eqnarray}\label{s4.10}
\begin{split}
   \Psi (x,t)=&\int_{-\infty}^{x}           [ {(u-U )  } (y,t)  ]  \operatorname{ d }y\\
   &-\int_{-\infty}^{x} \left( v^{- \alpha  } v_y-{}{V}^{- \alpha  }V_y\right) (y,t)\operatorname{ d }y \\
    :=& \psi (x,t)  +q(x,t )\leq \psi (x,t)  + C_{ } |\phi_{ x} (x,t)|,\\
     \psi (x,t)=& \Psi (x,t)  -q(x,t )\leq \Psi (x,t)  + C_{ }| \phi_{ x} (x,t)|.
     \end{split}
\end{eqnarray}
one have (\ref{s4.09}) from  (\ref{s4.10}) immediately.   The  estimates (\ref{s4.07}) and (\ref{s4.08}) can  be found in  \cite{hh2020} and   \cite{mm1999}  respectively. Thus the proof is completed.
\end{proof}
It is worth  to point out that the initial data $ (v_0, h_0)(x)   $  should satisfy the following equation
\begin{align}\label{4.8}
\begin{split}
0=&\int_{-\infty}^{\infty}\left(
\begin{array}{cccc}
{}{v}_0(x) -    V(x,0;\beta_{1}, \beta_{2};\beta)  \\
{}{h}_0(x) -    H(x,0;\beta_{1}, \beta_{2};\beta)
\end{array}
\right)
\operatorname{d}x.
\end{split}&
\end{align}
Here $H(x,0;\beta_{1}, \beta_{2};\beta)=U(x,0;\beta_{1}, \beta_{2};\beta)-[V(x,0;\beta_{1}, \beta_{2};\beta)]^{-(\alpha+1)}[V(x,0;\beta_{1}, \beta_{2};\beta)]_{x}$.   By directly calculate,   we know $({\ref{4.8}})$ is equivalent to   $I_i(\beta_1,\beta_2;\beta)=0, i=1,2.$

\begin{lma}\label{lemma4.2}
Under the same conditions of     Proposition \ref{prposition3.1}, it holds that
\begin{flalign}
 \|W\|_{2}\leq C_{ }  e^{-C_-\beta}e^{-c' t},
\end{flalign}
where $ C_{ } ,C_-,c' $ are constants independent of t.
 \end{lma}
\begin{proof}
 \begin{align}
\begin{split}
|W|=&\left|\frac{U_{1x}}{V_1^{\alpha+1}}+\frac{U_{2x}}{V_2^{\alpha+1}}-\frac{U_{x}}{V^{\alpha+1}}+p(V)+p(v_m)-p(V_{1})-p(V_2) \right|\\
=& \left|\left(\frac{U_{1x}}{V_1^{\alpha+1}}-\frac{U_{1x}}{V^{\alpha+1}}\right)+\left(\frac{U_{2x}}{V_2^{\alpha+1}}-\frac{U_{2x}}{V^{\alpha+1}}\right)\right|\\
&+\left|\left(p(V)-p(V_{1})\right)+\left(p(v_m)-p(V_2)\right)\right|\\
=&\left| U_{1x}\left(\frac{1}{V_1^{\alpha+1}}-\frac{1}{V^{\alpha+1}}\right)+    U_{2x}  \left(\frac{1}{V_2^{\alpha+1}}-\frac{1}{V^{\alpha+1}}\right)\right|\\
&+\left|\left(p(V_{1}+V_{2}-v_{m})-p(V_{1})\right)+\left(p(v_m)-p(V_2)\right)\right|\\
\leq&C\{ |(V_2-v_m)|+|U_{2x}|\}.
\end{split}&
\end{align}
By $(\ref{2.3}) $,  we get
\begin{equation}\label{4.11}
 \left|\frac{\partial^{j}U_{2}}{\partial x^{j}}\right| ,\left|\frac{\partial^{j}(V_{2}-v_m)}{\partial x^{j}}\right|    \leq C  |V_2-v_m|, \forall j\in \mathbb{N}.
\end{equation}
On the other hand,  in the same way,   it is still true to replace ($V_2,U_2$) with ($V_1,U_1$) in (\ref{4.11}).  We get $\left|\frac{\partial^{n}W}{\partial x^{n}}\right| \leq C  |V_i-v_m|, i=1,2; \forall n\in \mathbb{N}.$
By (\ref{2.9}) and (\ref{2.10}), we know   $ {\beta}_{1},{\beta}_{2} $  are bounded. {If we choose $\beta> 3 \max\{|\beta_{1}|,|\beta_{2}| \}$, for $n=0,1 $, it follows that:}
\begin{align*}
\begin{split}
\int_{-\infty}^{\infty} \left|\frac{\partial^{n} W}{\partial{x}^{n}}\right|^{2}  \operatorname{d}x=& \int_{-\infty}^{  \frac{\beta}{2}}\left|\frac{\partial^{n} W}{\partial{x}^{n}}\right|^{2}  \operatorname{d} x+ \int_{  \frac{\beta}{2}}^ { \infty}\left|\frac{\partial^{n} W}{\partial{x}^{n}}\right|^{2} \operatorname{d}x\\
 \leq &C \int_{-\infty }^{ \frac{\beta}{2}}   |V_{2}( x-s_{2}t+\beta_{2}-\beta)-v_{m}|^{2} \operatorname{d}x\\
 &+C \int_{ \frac{\beta}{2}}^ { \infty }  |V_{1}( x-s_{1}t+\beta_{1} )-v_{m}|^{2} \operatorname{d}x\\
 \leq &C \chi_{2}^{2} \int_{-\infty }^{ \frac{\beta}{2}}    \exp[ 2c_{2}( x-s_{2}t+\beta_{2}-\beta)]  ^{} \operatorname{d}x\\
 &+C \chi_{1}^{2} \int_{ \frac{\beta}{2}}^ { \infty }   \exp[ -2c_{1}  ( x-s_{1}t+\beta_{1} )]  ^{} \operatorname{d}x\\
=&  C \chi_{2}^{2} e^{- 2c_{2} s_{2} t}      \quad \frac{e^{c_{2}\left( 2\beta_{2}- \beta\right) }}{2c_{2}} + C   \chi_{1}^{2}   e^{2c_{1} s_{1} t}  \quad \frac{e^{-c_{1}\left(  {\beta} +2\beta_{1}\right)}}{ 2c_{1}} \\
\leq&C   e^{- 2c_{2} s_{2} t}        {e^{ - \frac{c_{2} }{3}\beta }} + C   e^{2c_{1} s_{1} t}    {e^{  \frac{- c_{1} }{3}{\beta}   }}.  \\
\end{split}&
\end{align*}
We have used Lemma \ref{lemma2.1} in the second inequality. Setting $c':=\min (-c_{1} s_{1},  c_{2} s_{2}     ); C_{-} := \frac{1}{6}\min (c_{1}  ,  c_{2}       ),         $ we get the proof of the Lemma.
\end{proof}
\begin{lma}\label{lemma4.3}
Under the same assumptions of Proposition \ref{prposition3.1}, it holds that
\begin{align*}
\begin{split}
&\|(\phi,\Psi)\|_{ }^2(t)+ \int_0^t\int_{-\infty}^{\infty}   \left(\frac{1}{p'(V)}\right)_t \Psi^2\operatorname{d} x\operatorname{ d} t+\int_0^t \| \phi_x\|^2\operatorname{ d} t\\
\leq& C_{}\|(\phi_0,\Psi_0)\|_{ }^2+C{\delta}  \int_0^t \| \phi_{xx}\|^{2}\operatorname{  d}t + C_{} e^{-C_-\beta}.
\end{split}
\end{align*}
\end{lma}
\begin{proof}
Multiply $(\ref{4.6})_1 $ and $(\ref{4.6})_2$ by $\phi$ and $\frac{\Psi}{-p'(V)}$, respectively, sum them up, and  integrate result with  respect to $t$ and $x$ over $ [0,t]\times \mathbb{R} $. We have
\begin{align}
\begin{split}
&\frac{1}{2}\int_{-\infty}^{\infty}      \left(\phi^2-\frac{\Psi^2}{p'(V)}\right)   \operatorname{d}x
+\int_0^t\int_{-\infty}^{\infty}\left\{\frac{1}{2}\left(\frac{1}{p'(V)}\right)_t \Psi^2
+       \frac{\phi_{x}^2}{V^{\alpha+1}}\right\}   \operatorname{d} x \operatorname{ d} t \\
=&\int_0^t\int_{-\infty}^{\infty}       G_{ }\phi  \operatorname{d}x \operatorname{ d} t
+ \int_0^t\int_{-\infty}^{\infty}          \frac{p({v}|V)\Psi}{p'(V)}  \operatorname{d}x \operatorname{ d} t
\\
&+ \int_0^t\int_{-\infty}^{\infty}        W\frac{\Psi}{p'(V)}  \operatorname{d }x \operatorname{ d} t+\frac{1}{2}    \int_{-\infty}^{\infty}      \left(\phi^2-\frac{\Psi^2}{p'(V)}\right)\Big|_{t=0}   \operatorname{d }x =: \sum_{i=1}^4 A_i.
\label{4.12}\end{split}&
\end{align}
Utilize to Lemma \ref{lemma4.1}, we can get
\begin{align}
\begin{split}
&|A_1+A_2|\\
\leq& C \left(\int_0^t\int_{-\infty}^{\infty}   \left|  \phi_{xx}^{2} \phi \right| +  \left|    \phi_x \phi_{xx}  \phi\right|  +     \left| \Psi \phi_x^2 \right|  \operatorname{d}x \operatorname{d}t\right)\\
\leq& C \int_0^t    \|\phi\|_{\mathbb{L}^\infty}    \int_{-\infty}^{\infty}   \left|\phi_x^{2}+ \phi_{xx}^{2}  \right| \operatorname{d}x \operatorname{d}t+ C \int_0^t      \|\Psi\|_{\mathbb{L}^\infty}   \int_{-\infty}^{\infty}   \phi_x^2  \operatorname{d}x \operatorname{d}t\\
\leq&   C (\|\phi\|_{2}+  \|\psi\|_{1}   ) \int_0^t        \|\phi_x\|^2 +\|\phi_{xx}\|^2  \operatorname{  d}t\\
\leq & C \delta \int_0^t        \|\phi_x\|^2 +\|\phi_{xx}\|^2  \operatorname{  d}t.
\end{split}&
\end{align}
With the help of Lemma \ref{lemma4.2}, one has
\begin{align}
\begin{split}
|A_3|\leq& C  \int_0^t\int_{-\infty}^{\infty}      \left| W \Psi \right|  \operatorname{d}x\operatorname{d}t\leq C  \int_0^t     \|W\|   \|\Psi\|        \operatorname{d}t\leq  C  \delta     e^{-C_- \beta}.
\end{split}&
\label{4.14}\end{align}
Taking $\delta$  sufficient small, using (\ref{4.12})-(\ref{4.14}), we get Lemma \ref{lemma4.3}.
\end{proof}
\begin{lma}\label{lemma4.4}
Under the same assumptions of Proposition \ref{prposition3.1}, it holds that
\begin{eqnarray*}
\|(\phi,\Psi)(t)\|_{ 1}^2+\int_0^t\|  \phi_x  \|_{ 1}^2\operatorname{d}t\leq C_{ }\|(\phi_0,\Psi_0)\|_{ 1}^2 + C_{ } e^{-C_-\beta}.
\end{eqnarray*}
\end{lma}
\begin{proof}
Multiply $ (\ref{4.6})_1 $ and $  (\ref{4.6})_2 $ by $-\phi_{xx}$  and $\frac{\Psi_{xx}}{p'(V)}$, respectively and sum over the result, intergrade the result with  respect to $t$ and $x$ over $ [0,t]\times \mathbb{R}$. We have
\begin{align}
\begin{split}
& \frac{1}{2}  \int_{-\infty}^{\infty}  \left(\phi_x^2-\frac{\Psi_x^2}{p'(V)}\right)    \operatorname{d}x    + \int_0^t\int_{-\infty}^{\infty} \left\{ \frac{1}{2}  \left(\frac{1}{p'(V)}\right)_t\Psi_x^2    + \frac{\phi^2_{xx}}{V^{\alpha+1}} \right\}\operatorname{d}x\operatorname{d}t \\
=&\frac{1}{2}  \int_{-\infty}^{\infty}  \left(\phi_x^2-\frac{\Psi_x^2}{p'(V)}\right)\Big|_{t=0}    \operatorname{d}x  \\
&-\int_0^t\int_{-\infty}^{\infty}  \left[ G -(\alpha+1)\frac{V_x }{V^{\alpha+2}}  \phi_x \right]\phi_{xx}   \operatorname{d}x\operatorname{d}t  \\
 &-  \int_0^t\int_{-\infty}^{\infty}     \left(\frac{1}{p'(V)}\right)_{x}  p'(V) \Psi_{x} \phi_{x}  \operatorname{d}x\operatorname{d}t + \int_0^t\int_{-\infty}^{\infty}  \frac{\Psi_{x}W_{  x}}{p'(V)}  \operatorname{d}x\operatorname{d}t\\
  &+  \int_0^t\int_{-\infty}^{\infty}      \frac{1}{p'(V)} p(v|V)_{x}\Psi_{x} \operatorname{d}x\operatorname{d}t\\
   =:& \frac{1}{2}  \int_{-\infty}^{\infty}  \left(\phi_x^2-\frac{\Psi_x^2}{p'(V)}\right)\Big|_{t=0}    \operatorname{d}x +\sum_{i=1}^{4} B_i.
\end{split}&
\label{4.15}\end{align}{}
Now we estimate $B_i$ term by term. The Cauchy inequality indicates that
\begin{align}
\begin{split}
|B_1|&    \leq     C  \int_0^t\int_{-\infty}^{\infty}   (|\phi_{xx}\phi_x|+| \phi^{2}_x|)|{ \phi_{xx}}| +      | \phi_x     \phi_{xx} |               \operatorname{d}x\operatorname{d}t .\\
  &\leq (C \delta + \varepsilon )\int_0^t  \|\phi_{xx}\|^{2}  \operatorname{d}t  + C_{\varepsilon}   \int_0^t   \|\phi_{x}\|^2\operatorname{d}t,\\
\end{split}&
\end{align}
and
\begin{align}
\begin{split}
B_2  \leq&  \min(-s_1;s_2) \int_0^t \int_{-\infty}^{\infty}\left|\left(\frac{1}{4 p'(V)}\right)_x\right|\Psi_x ^2\operatorname{d}x\operatorname{d}t \\
&+ C\int_0^t\int_{-\infty}^{\infty}\left| \left(\frac{1}{ p'(V)}\right)_x [p'(V)]^2\right|\phi_x ^2\operatorname{d}x\operatorname{d}t\\
\leq& \int_0^t\int_{-\infty}^{\infty}\left(\frac{1}{4 p'(V)}\right)_t \Psi_x ^2\operatorname{d}x\operatorname{d}t +C \int_{0}^{t}   \|\phi_x\| ^2  \operatorname{d}t\\
\end{split}&
\label{4.17}\end{align}
The last inequality is base on the following inequality
\begin{align*}
\begin{split}
 \left(\frac{1}{p'({}{V})}\right)_t =& \left(p'({}{V})\right)^{-2} p''({}{V}) (-{}{V}_{1}'(-s_1)+{}{V}'_{2}( s_2))  \\
\geq&  \left(p'({}{V})\right)^{-2} p''({}{V}) |{}{V}_{x} | \min(-s_1; s_2) \\
=&\min(-s_1; s_2) \left|\left(\frac{1}{p'({}{V})}\right)_x\right|.
\end{split}&
\end{align*}
Making use of Lemma \ref{lemma4.2}, it follows that
\begin{align}
\begin{split}
|B_3|\leq \int_0^t\int_{-\infty}^{\infty}  \left| \frac{\Psi_{x}W_{  x}}{p'(V)}  \right|   \operatorname{d}x\operatorname{d}t \leq C  \int_0^t  \|W_{x}\|  \|\Psi_{x}\|   \operatorname{d}t \leq C \delta e^{-C_-\beta} .
\end{split}&
\end{align}
By $(\ref{s4.10})_{1}$ and the Sobolev inequality, we obtain
\begin{align}\label{4.19}
\begin{split}
|B_4|\leq & C\int_0^t\int_{0}^{\infty}     \left| (\phi_{x}\phi_{xx}+V_x\phi_{x}^{2}) \Psi_{x}       \right| \operatorname{d}x \operatorname{d} t    \\
\leq&   C\int_0^t\int_{0}^{\infty}     \left| (\phi_{x}\phi_{xx}+V_x\phi_{x}^{2}) \psi_{x}       \right|   \operatorname{d}x\operatorname{d} t   \\
& + C\int_0^t\int_{0}^{\infty}          \left| (\phi_{xx}\phi_{xx}+V_x\phi_{x}\phi_{xx} ) \phi_{x}       \right| \operatorname{d}x\operatorname{d} t \\
\leq&   C (\|\phi\|_{2}+  \|\psi\|_{2}   ) \int_0^t        \|\phi_x\|^2 +\|\phi_{xx}\|^2  \operatorname{  d}t \\
\leq&   C\delta \int_0^t    (\|\phi_{xx}\|^2+\|\phi_{x}\|^2 )    \operatorname{ d} t.
\end{split}
\end{align}
From (\ref{4.15})-(\ref{4.19}), we get
\begin{align*}
\begin{split}
& \frac{1}{2} \int_{-\infty}^{\infty}  \left(\phi_x^2-\frac{\Psi_x^2}{p'(V)}\right)    \operatorname{d}x+ \frac{1}{4}\int_0^t\int_{-\infty}^{\infty}    \left(\frac{1}{p'(V)}\right)_t\Psi_x^2\operatorname{d}x\operatorname{d}t  +\int_0^t\int_{-\infty}^{\infty}  \frac{\phi^2_{xx}}{V^{\alpha+1}} \operatorname{d}x\operatorname{d}t \\
\leq&  (C +C\delta+C_{\varepsilon} )\int_0^t   \|\phi_{x}\|^2 \operatorname{d}t
 +(C\delta+\varepsilon )\int_0^t     \|\phi_{xx}\|^2  \operatorname{d}t\\
& +  C_{ } e^{-C_-\beta}+C_{ }    \left(   \|\phi_{0x}\|^2   +\|\Psi_{0x}\|^2 \right)   .
\end{split}&
\end{align*}
Choosing  $\varepsilon$  appropriately small and $\delta$ sufficient small,   together with Lemma \ref{lemma4.3}, we get  the proof of Lemma \ref{lemma4.4}.
\end{proof}{}
\begin{lma}\label{lemma4.5}
Under the same assumptions of Proposition \ref{prposition3.1}, it holds that
\begin{eqnarray*}
\int_0^t    \|\Psi_{x}(t)\|_{  }^2   \operatorname{d}t  \leq   C  \|(\phi_0,\Psi_0)\|_{ 1}^2 +    C e^{-C_-\beta}.
\end{eqnarray*}
\end{lma}
\begin{proof}
Multiplying $(\ref{4.6})_1 $  by $\Psi_{x}$  and make use of $(\ref{4.6})_2$, we get
\begin{align}
\begin{split}
\Psi_{x}^{2}=&(\phi\Psi_{x})_{t}+\{\phi[ p(v)-p(V)+W]\}_{x}-\phi_{x} (p( v)-p({V}{}))\\
&-\frac{\Psi_{x}\phi_{xx}}{V^{\alpha+1}}-\phi_x W -\Psi_{x} G +(\alpha+1) \frac{V_x}{V^{\alpha+2}}\phi_x\Psi_{x} .
\label{4.20}
\end{split}&
\end{align}
Integrating $ (\ref{4.20})$ with  respect to $t$ and $x$ over $ [0,t]\times \mathbb{R} $, we obtain that
\begin{align}
\begin{split}
& \int_{0}^{t}\int_{-\infty}^{\infty}  \Psi_{x}^{2}  \operatorname{d}x\operatorname{d}t\\
=&-\int_{-\infty}^{\infty} \phi\Psi_{x}|_{t=0}\operatorname{d}x-\int_{0}^{t}\int_{-\infty}^{\infty} \Psi_{x} G  \operatorname{d}x\operatorname{d}t\\
&+    \int_{0}^{t}\int_{-\infty}^{\infty} (\alpha+1)\frac{V_x}{V^{\alpha+2}} \Psi_{x} \phi_x\operatorname{d}x\operatorname{d}t
+\int_{-\infty}^{\infty}\phi\Psi_{x}\operatorname{d}x\\
&-\int_{0}^{t}\int_{-\infty}^{\infty} \frac{\Psi_{x}\phi_{xx}}{V^{\alpha+1}}\operatorname{d}x\operatorname{d}t-\int_{0}^{t}\int_{-\infty}^{\infty}\phi_{x}\left(p(v)-p(V)\right)       \operatorname{d}x\operatorname{d}t\\
&-\int_{0}^{t}\int_{-\infty}^{\infty} \phi_{x}W       \operatorname{d}x\operatorname{d}t\\
=&-\int_{-\infty}^{\infty} \phi\Psi_{x}|_{t=0}\operatorname{d}x+\sum_{i=1}^6 H_i.
\end{split}&
\label{4.21}\end{align}
We estimate $H_i$ term by term. By the Cauchy inequality, it follows that
\begin{align}
\begin{split}
H_1&\leq C \int_{0}^{t}  \|\phi_{x}\|_{L^{\infty}}  \int_{-\infty}^{\infty}    \Psi_{x}(|\phi_{xx}|+  |\phi_{x}|)         \operatorname{d}x\operatorname{d}t\\
&\leq  \varepsilon \int_{0}^{t}  \|\Psi_{x}\|^{2}  \operatorname{d}t+ C_{\varepsilon} \int_{0}^{t}   ( \|\phi_{xx}\|^{2} + \|\phi_{x} \|^{2}   )      \operatorname{d}t,
\end{split}&
\end{align}
In addition, it is straightforward to imply that
\begin{align}
\begin{split}
&H_2+H_3+H_4+H_5\\
\leq &  \| (\phi^{}, \Psi_{x}) \|^{2}+\varepsilon\int_{0}^{t} \|\Psi_{x}\|^{2}\operatorname{  d}t+C_{\varepsilon}\int_{0}^{t}( \|\phi_{xx}\|^2       + \|\phi_{x}\|^2 )\operatorname{  d}t.
\label{c4}\end{split}&
\end{align}
Making use of Lemma \ref{lemma4.2}, we have
\begin{align}
\begin{split}
H_6&=\int_{0}^{t}\int_{-\infty}^{\infty}\phi_{x}W       \operatorname{d}x\operatorname{d}t \leq  \int_{0}^{t}  \| W\|       \|\phi_{x}\|      \operatorname{d}t
\leq C \delta e^{- C_-\beta}.
\label{4.24}
\end{split}&
\end{align}
Collecting   (\ref{4.21})-(\ref{4.24}) and using Lemma \ref{lemma4.4}, we get the Lemma \ref{lemma4.5}.
\end{proof}
Combining Lemma \ref{lemma4.3}-Lemma \ref{lemma4.5}, we  obtain the following low order estimates
\begin{align*}
\|(\phi,\Psi)\|_{ 1}^2(t)+ \int_0^t  \| \Psi_x\|^2  \operatorname{d}t+\int_0^t \| \phi_x\|_{1}^2 \operatorname{d}t\leq C_{}\|(\phi_0,\Psi_0)\|_{ 1}^2 + C_{} e^{-C_-\beta},
\end{align*}
with the help of $(\ref{s4.09})$, which can be rewritten by the variables $\phi$ and $\psi$ as
\begin{lma}\label{lemma4.6}
Under the same assumptions of Proposition \ref{prposition3.1}, it holds that
\begin{align*}
\begin{split}
&(\| \phi \|_{ 1}^2 + \| \psi \| ^2)(t)+\int_0^t  \| \psi_x\|^2  \operatorname{d}t+\int_0^t \| \phi_x\|_{1}^2 \operatorname{d}t\leq C_{}\|     \phi_0 \|_{2}^2 +C_{}\| \psi_0 \|_{ 1}^2 + C_{} e^{-C_-\beta}.
\end{split}
\end{align*}
\end{lma}
\subsection{High Order Estimate.}
We turn to the original equation \eqref{3.3} to study the higher order estimates.
\begin{lma}\label{lemma4.7}
Under the same assumptions of Proposition \ref{prposition3.1}, it holds that
\begin{align}\label{4.25}
\begin{split}
&  \| \psi_{x} \|  ^2 (t)  +\int_0^t  \| \psi_{xx}\| ^2  \operatorname{d}t\leq C_{}\|     \phi_0 \|_{2}^2 +C_{}\| \psi_0 \|_{ 1}^2 + C_{} e^{-C_-\beta}.
\end{split}
\end{align}
\end{lma}
\begin{proof}  Multiplying   $(\ref{3.3})_{2}$ by $-\psi_{x x}$,  integrating the result with  respect to $t$ and $x$ over $ [0,t]\times \mathbb{R} $ gives
\begin{align}\label{4.26}
\begin{split}
&   \frac{1}{2}\| \psi_{x}\|^{2}  (t)
+\int_{0}^{t}\int_{-\infty}^{\infty} \frac{{\psi_{x x}^{2}}}{V^{\alpha+1}} \operatorname{d}x \operatorname{d}t \\
=&\frac{1}{2}\| \psi_{0x}\|^{2} +\int_{0}^{t}\int_{-\infty}^{\infty} W \psi_{x x} \operatorname{d}x \operatorname{d}t\\
&-\int_{0}^{t}\int_{-\infty}^{\infty} f(V,U_x) \phi_{x} \psi_{x x} \operatorname{d}x \operatorname{d}t   -\int_{0}^{t}\int_{-\infty}^{\infty} F \psi_{x x} \operatorname{d}x \operatorname{d}t
   \\
=:&\frac{1}{2}\| \psi_{0x}\|^{2} +\sum_{i=1}^3 M_i.
\end{split}
\end{align}
Making use of Lemma \ref{lemma4.2}, we have
\begin{align} \label{4.27}
\begin{split}
M_{1}  &\leq \varepsilon \int_{0}^{t} \|\psi_{x x}\|^{2}\operatorname{  d}t+ C_{\varepsilon}\int_{0}^{t}  \|W\|^{2}\operatorname{  d}t. \\
 &\leq \varepsilon \int_{0}^{t} \|\psi_{x x}\|^{2}\operatorname{  d}t+  C_{} e^{-C_-\beta}.
\end{split}
\end{align}
The Cauchy inequality implies that
\begin{align}\label{4.28}
M_{2}  \leq \varepsilon \int_{0}^{t} \|\psi_{x x}\|^{2}\operatorname{  d}t+ C_{\varepsilon}\int_{0}^{t}  \|\phi_{x}\|^{2}\operatorname{  d}t.
\end{align}
By $(\ref{s4.08})_{1}$ and the Sobolev inequality, yields
\begin{align} \label{4.29}
\begin{split}
M_{3} &\leq C \int_{0}^{t}\int_{-\infty}^{\infty}\left(\left|\phi_{x}\right|^{2}  +\left|\phi_{x}\right|\left|\psi_{x x}\right|\right)\left|\psi_{x x}\right| \operatorname{d} x \operatorname{d}t \\
&\leq C \int_{0}^{t}\int_{-\infty}^{\infty}\left|\phi_{x}\right|\left(\left|\phi_{x}\right|^{2}+\left|\psi_{x x}\right|^{2}\right) \operatorname{d}x\operatorname{d}t   \\
&\leq C   \delta \int_{0}^{t} \left(\left\|\phi_{x}\right\|^{2}+\left\|\psi_{x x}\right\|^{2}\right) \operatorname{  d}t.
\end{split}
\end{align}
Substituting (\ref{4.27})-(\ref{4.29})  into $(\ref{4.26}) $ and using Lemma \ref{lemma4.6}, we obtain (\ref{4.25}).
\end{proof}
\begin{lma}\label{lemma4.8}
Under the same assumptions of Proposition \ref{prposition3.1}, it holds that
\begin{align}\label{4.30}
\begin{split}
&\| \phi_{xx} \| ^2   +\int_0^t  \| \phi_{xx}\|_{}^2  \operatorname{d}t \leq C_{}\|     \phi_0 \|_{2}^2 +C_{}\| \psi_0 \|_{ 1}^2 + C_{} e^{-C_-\beta}+ C\delta \int_{0}^{t}\left\|\psi_{xx x} \right\|^2  \operatorname{d}t.
\end{split}
\end{align}
\end{lma}

\begin{proof}
 Differentiating    $(\ref{3.3})_{1}$ with  respect to $x$,  using  $(\ref{3.3})_{2}, $ we have
\begin{align}\label{4.31}
\begin{split}
\frac{  \phi_{x t}}{V^{\alpha+1}}+f(V,U_x) \phi_{x}=\psi_{t}-F+W.
\end{split}
 \end{align}
Differentiating    $(\ref{4.31})$ in respect of $x$ and multiplying  the derivative  by $\phi_{x x}$,  integrating the result in respect of $t$ and $x$ over $ [0,t]\times \mathbb{R} $, using (\ref{2.3}), one has
\begin{align} \label{4.32}
\begin{split}
 &\frac{1}{2} \int_{-\infty}^{\infty}  \frac{ \phi_{x x}^{2}}{  V^{  \alpha+1  }}  \operatorname{d  }x
 +\int_{0}^{t}\int_{-\infty}^{\infty}\left(f(V,U_x)+  \frac{\alpha+1}{2} \frac{  U_x  }{ V^{\alpha+2}}\right) \phi_{x x}^{2}\operatorname{d}x \operatorname{d}t  \\
 =& \frac{1}{2} \int_{-\infty}^{\infty}  \frac{ \phi_{x x}^{2}}{  V^{  \alpha+1  }}  \Big|_{t=0}\operatorname{d  }x-\int_{-\infty}^{\infty}   \psi_{x } \phi_{x x}    \Big|_{t=0}\operatorname{d  }x+\int_{-\infty}^{\infty}   \psi_{x } \phi_{x x}   \operatorname{d  }x  \\
&+\int_{0}^{t}\int_{-\infty}^{\infty}W_{x}\phi_{x x}\operatorname{d}x\operatorname{d}t
+\int_{0}^{t}        \| \psi_{x x}\|^{2}\operatorname{    d}t-\int_{0}^{t}\int_{-\infty}^{\infty}F_{x} \phi_{x x}\operatorname{d}x\operatorname{d}t
  \\
&+(\alpha+1)\int_{0}^{t}\int_{-\infty}^{\infty}\frac{{} V_{x}}{V^{\alpha+2}} \psi_{x x} \phi_{x x}\operatorname{d}x\operatorname{d}t-\int_{0}^{t}\int_{-\infty}^{\infty}f(V,U_x)_{x} \phi_{x} \phi_{x x}\operatorname{d}x\operatorname{d}t   \\
=:&\frac{1}{2} \int_{-\infty}^{\infty}  \frac{ \phi_{x x}^{2}}{  V^{  \alpha+1  }}  \Big|_{t=0}\operatorname{d }x-\int_{-\infty}^{\infty}   \psi_{x } \phi_{x x}    \Big|_{t=0}\operatorname{d  }x+\sum_{i=1}^6 N_i.
\end{split}
\end{align}
By $U_x<0$ and (\ref{3.5}), one has
\begin{align} \label{4.33}
\begin{split}
&f(V,U_x)+  \frac{\alpha+1}{2} \frac{  U_x  }{ V^{\alpha+2}}\\
=&-p'(V )-  \frac{\alpha+1}{2} \frac{  U_x  }{ V^{\alpha+2}}\geq - \max \{p'(v_-),p'(v_+)\}>0.
\end{split}
\end{align}
The Cauchy inequality yields
\begin{align}
N_{1}\leq  \varepsilon  \|\phi_{x x}\|^{2}   +C_\varepsilon    \|\psi_{x }\|  ^{2}.
\end{align}
 Similar to (\ref{4.27}), we get
\begin{align}
N_{2}  \leq \varepsilon       \int_{0}^{t} \|\phi_{x x}\|^{2} \operatorname{  d}t    +C_\varepsilon   e^{-C_-\beta}.
\end{align}
$ N_{3}  $ can be controlled by (\ref{4.25}). Using  $(\ref{s4.08})_{2}$, and Cauchy inequality,   we have
\begin{align*}
\begin{split}
|N_{4}| \leq &  \varepsilon \int_{0}^{t}  \|\phi_{x x}\|^{2} \operatorname{  d}t  +    C_{\varepsilon}   \int_{0}^{t}\left\|F_{ x}\right\|^{2} \operatorname{  d}t. \\
\leq &  \varepsilon \int_{0}^{t}  \|\phi_{x x}\|^{2} \operatorname{  d}t  +  C_{\varepsilon}  \delta     \int_{0}^{t}  \left(\left\|\phi_{x}\right\|_{1}^{2}+\left\|\psi_{x}\right\|_{2}^{2}\right)\operatorname{  d}t.
\end{split}
\end{align*}
The Cauchy inequality yields
\begin{align}
\begin{split}
|N_{5}| \leq C & \int_{0}^{t}\int_{-\infty}^{\infty} \left|\frac{{} V_{x}}{V^{\alpha+2}} \psi_{x x} \phi_{x x}\right| \operatorname{d}x \operatorname{d}t\leq \varepsilon \int_{0}^{t}  \|\phi_{x x}\|^{2} \operatorname{  d}t  +    C_{\varepsilon}   \int_{0}^{t}\left\|\psi_{x x}\right\|^{2} \operatorname{  d}t.
\end{split}
\end{align}
With the help of
$$f(V,U_x)_{x}= -p^{\prime\prime}(V)V_{x}    - (\alpha+1)\frac{{} U_{xx}}{V^{\alpha+2}}   + (\alpha+1)(\alpha+2)\frac{{} U_{x}}{V^{\alpha+3}}V_{x}   < C, $$
one gets
\begin{align} \label{4.38}
\begin{split}
|N_{6}| \leq &  \varepsilon \int_{0}^{t}  \|\phi_{x x}\|^{2} \operatorname{  d}t  +    C_{\varepsilon}   \int_{0}^{t}\left\|\phi_{ x}\right\|^{2} \operatorname{  d}t. \\
\end{split}
\end{align}
Choosing  $\varepsilon$  small,  substituting  (\ref{4.33})-(\ref{4.38})  into (\ref{4.32})   and  using  Lemma \ref{lemma4.6}, Lemma \ref{lemma4.7},    we have (\ref{4.30}).
\end{proof}

On the other hand, differentiating the second equation of (\ref{3.3}) with respect to $x$, multiplying the derivative by $-\psi_{x x x}$, integrating the resulting equality over $[0, \infty)  \mathbb{\times}[0, t]$, using Lemma \ref{lemma4.6}-Lemma \ref{lemma4.8}, we can get the highest order estimate in the same way, which is listed as follows  and the proof is omitted.

\begin{lma}\label{lemma4.9}
Under the same assumptions of Proposition \ref{prposition3.1}, it holds that
\begin{align}\label{4.39}
\begin{split}
&  \| \psi_{xx} (t)\| ^2+\int_0^t  \| \psi_{xxx}\| ^2  \operatorname{d}t \leq   C_{}\|   (  \phi_0 , \psi_0 )\|_{ 2}^2 + C_{} e^{-C_-\beta} .
\end{split}
\end{align}
\end{lma}
Finally, Proposition \ref{prposition3.1} is obtained by Lemma \ref{lemma4.5}-Lemma \ref{lemma4.9}.
\section{Proof of Theorem \ref{theorem}}\label{section5}
It is straightforward to imply   (\ref{2.13}) from  Lemma  \ref{lemma3.1}. It remains to show (\ref{2.14}). The following useful lemma will be used.
\begin{lma}(\cite{mn1985})\label{lemma5.1}
Assume that the function $f(t) \geq 0\in \mathbb{L}^1(0, +\infty) \cap  \mathbb{BV}(0, +\infty) $, then it holds that $f(t) \rightarrow0$ as $t \rightarrow \infty$.
\end{lma}
Let us turn to the system   (\ref{3.3}). Differentiating  (\ref{3.3})$_1$   with respect to $x$, multiplying the
resulting equation by $\phi_{  x}$ and integrating it with respect to $ { x}$ on $(-\infty,\infty)$, we have
\begin{equation*}
 \left|\frac{\operatorname{d}}{\operatorname{d}t}\left(\|\phi_{ x}\|^{2}\right)\right|\leq C(\|\phi_{x}\|^{2} +\|\psi_{xx}\|^{2})  .
\end{equation*}
With the aid of Lemma  {\ref{lemma3.1}}, we have
\begin{equation*}
\int_{-\infty}^{\infty} \left|\frac{\operatorname{d}}{\operatorname{d}t}\left(\|\phi_{xx}\|^{2}\right)\right| \operatorname{d}t \leq C,
\end{equation*}
which implies $\|\phi_{ x}\|^{2}\in \mathbb{L}^1(0, +\infty) \cap  \mathbb{BV}(0, +\infty)$. By Lemma \ref{lemma5.1}, we have
 \begin{equation*}
    \|\phi_{ x}\|\rightarrow0 \quad   \text{as} \quad   t\rightarrow+\infty.
 \end{equation*}
Since $\|\phi_{xx}\|$ is bounded, the Sobolev inequality implies that
 \begin{eqnarray*}
 \|{}{v}-V\|_{\infty}^{2}=\|\phi_{x}\|_{\infty}^{2} \leq 2\| \phi_{x}(t)\|_{}   \| \phi_{xx}(t)\|_{} \rightarrow  0.
\end{eqnarray*}
 Similarly, we have
 \begin{eqnarray*}
 \|{}{u}-U\|_{\infty}^{2}=\|\psi_{x}\|_{\infty}^{2} \leq 2\| \psi_{x}(t)\|_{}   \| \psi_{xx}(t)\|_{} \rightarrow  0.
\end{eqnarray*}
Therefore, the proof of Theorem \ref{theorem} is completed.

\end{document}